\documentclass{amsart}
\usepackage{amssymb}
\usepackage{amscd}
\usepackage{amsfonts}
\usepackage{latexsym}
\usepackage[all]{xy}
\usepackage{verbatim}
\usepackage{hyperref}
\usepackage{mathenv}
\usepackage{mathbbol}

\numberwithin{equation}{section}

\newtheorem{theorem}{Theorem}[subsection]
\newtheorem{lemma}[theorem]{Lemma}
\newtheorem{proposition}[theorem]{Proposition}
\newtheorem{corollary}[theorem]{Corollary}
\theoremstyle{definition}

\newtheorem{example}[theorem]{Example}

\newtheorem{question}[theorem]{Question}
\newtheorem{remark}[theorem]{Remark}

\DeclareMathOperator{\Tr}{Tr}

\DeclareMathOperator{\id}{id}

\DeclareMathOperator{\FPdim}{FPdim}

\DeclareMathOperator{\Hom}{Hom}

\DeclareMathOperator{\Ve}{Vec}

\DeclareMathOperator{\Rep}{Rep}
\DeclareMathOperator{\Irr}{Irr}

\newcommand{\B}{\mathcal{B}}
\newcommand{\C}{\mathcal{C}}

\newcommand{\Z}{\mathcal{Z}}

\newcommand{\A}{\mathcal{A}}

\renewcommand{\O}{\mathcal{O}}

\newcommand{\be}{\mathbf{1}}

\renewcommand{\be}{\mathbf{1}}

\newcommand{\BZ}{{\mathbb Z}}
\newcommand{\BR}{{\mathbb R}}
\newcommand{\BQ}{{\mathbb Q}}

\newcommand{\BC}{{\mathbb{C}}}

\newcommand{\bt}{\boxtimes}
\newcommand{\ot}{\otimes}

\begin{document}
\title{Remarks on global dimensions of fusion categories}

\author{Victor Ostrik}
\address{Department of Mathematics,
University of Oregon, Eugene, OR 97403, USA}
\address{Laboratory of Algebraic Geometry,
National Research University Higher School of Economics, Moscow, Russia}
\email{vostrik@uoregon.edu}

\begin{abstract} 
We give a nontrivial lower bound for global dimension of a spherical fusion category.
\end{abstract}

\date{\today} 
\maketitle  

\section{Introduction}

\subsection{}
This paper is a contribution to the theory of fusion categories over the field of complex numbers. 
Recall (see \cite{ENO, EGNO}) that a fusion category is a semisimple rigid tensor category with
finitely many simple objects, finite dimensional $\Hom-$spaces, and with simple unit object.
The simplest example is the category of finite dimensional vector spaces $\Ve$. A basic invariant
of a fusion category $\C$ is its {\em global dimension} $\dim(\C)\in \BR$ introduced by M.~M\"uger
in \cite{Mu}, see Section \ref{gldim} below. We consider the global dimension as a real valued
function on the set of equivalence classes of fusion categories. Here is our first result.

\begin{theorem}[see Theorem \ref{ffib}] The fibers of the map $\C \mapsto \dim(\C)$ are finite. 
In other words there is at most finitely many fusion categories (up to a tensor equivalence) of a fixed global dimension.
\end{theorem}

It seems natural to ask next about properties of the image of the global dimension, that is of the following subset of $\BR$:
\begin{equation} \label{fX}
X:=\{ \dim(\C)\; |\; \C \; \mbox{is a fusion category}\}
\end{equation}
It is known that $X$ consists of totally real algebraic integers (and so $X$ is countable), is invariant under the Galois group, and
is contained in the interval $[1,\infty )$, see Section \ref{gldim}. However other properties of the set $X$
are not well understood. For instance it is an open question (asked in \cite[p. 596]{ENO}) whether the set $X$ is discrete. A weaker version of this question (also asked in {\em loc. cit.} and still open) is whether point $1=\dim(\Ve)\in X$ is isolated. An expected positive answer to this question was called {\em categorical property T} in {\em loc. cit.} by analogy with the famous Kazhdan's property T in the group theory. The main result
of this paper is a slightly weaker version of the categorical property T. Namely, recall that a fusion
category $\C$ is {\em spherical} if it admits a spherical structure, see Section \ref{pivot}. All currently known
fusion categories are spherical and it is expected (see e.g. \cite[Conjecture 2.8]{ENO}) that this is always the case. This
conjecture is of great importance for the theory of fusion categories, but it seems to be quite difficult.
We consider
the following subset of the set $X$:
\begin{equation}
X_s:=\{ \dim(\C)\; |\; \C \; \mbox{is a spherical fusion category}\}
\end{equation}
Clearly, the set $X_s$ can be considered as a spherical counterpart of the set $X$ above.

\begin{theorem}[see Theorem \ref{gdim} (i)] \label{main}
Let $\C$ be a spherical fusion category which is not equivalent to $\Ve$. Then
\begin{equation} 
\dim(\C)>\frac43
\end{equation} 
In particular point $1=\dim(\Ve)\in X_s$ is isolated.
\end{theorem}

The constant $\frac43$ in Theorem \ref{main} is not optimal. In the Appendix (joint with P.~Etingof)
we show that for a spherical fusion category $\C \not \simeq \Ve$ we have
$$\dim(\C)\ge \frac{5-\sqrt{5}}2=\dim(YL)\approx 1.381966$$ 
where $YL$ is the (non-unitary) Yang-Lee category, see e.g. \cite[Exercise 4.10.7]{EGNO}. 
Note that the point $\dim(YL)\in X_s$ is isolated by Corollary \ref{kiz2} or by more precise Proposition \ref{mainpro}.

We will also prove Theorem \ref{codeg} which extends Theorem \ref{main} to other numerical
invariants of the category $\C$, so called {\em formal codegrees}, see Section \ref{fcgen}.

%Theorem \ref{main} will be proved in Section \ref{bound} as a consequence of more general Theorem \ref{codeg}.

\subsection{Open questions}
\begin{question} Find a counterpart of Theorem \ref{main} for a not necessarily spherical fusion
categories.
\end{question}

%The constant $\frac43$ in Theorem \ref{main} is almost surely non optimal. Thus we have

%\begin{question} \label{optim} 
%Find optimal constant for the statement of Theorem \ref{main}.
%\end{question}

%We expect that the answer to question \ref{optim} is $\frac{5-\sqrt{5}}2=\dim(YL)\approx 1.381966$ where
%$YL$ is the (non-unitary) Yang-Lee category, see e.g. \cite[Exercise 4.10.7]{EGNO}. 

%\begin{question} Is the point $\dim(YL)$ isolated in $X$ or $X_s$?
%\end{question}

It follows from Corollary \ref{kiz2} that any points $x\in X_s$ with $x<\sqrt{2}$ is isolated. Thus we have
\begin{question} \label{kiz2lim} 
Prove (or disprove) that $\sqrt{2}$ is not a limit point of $X_s$.
\end{question}

Note that it is clear that $\sqrt{2}\not \in X_s$ since $\sqrt{2}$ is not totally positive. 
Actually we don't know any point $x\in X_s$ which is less than $\sqrt{2}$ but greater than
$\dim(YL)$. Thus we ask

\begin{question} \label{bkiz2} 
Are there any points $x\in X_s$ with $\dim(YL)<x<\sqrt{2}$?
\end{question}

A more ambitious version of this question is to find a next point (or limit point) of $X_s$ after $\dim(YL)$. 
We suspect that the answer is given by the smallest root of the polynomial $t^3-14t^2+49t-49$ 
(approximately $1.84117$) which 
is the dimension of some fusion category $\C_3$ of rank 3 associated with quantum $\mathfrak{so}_3$
at $7$th root of 1, see \cite[4.3]{O4}.

\subsection{Acknowledgements} It is my great pleasure to thank Pavel Etingof for his comments which
allowed me to significantly improve the results of this note.
This work was partially supported by the NSF grant DMS-1702251. The study also has been funded by the Russian Academic
Excellence Project '5-100'.

%The global dimension is a map from the set of equivalence classes of fusion categories to $\BR$.
%The following question seems natural:

%\begin{question}
%Are the fibers of the map above finite? Equivalently, is it possible to have infinitely many pairwise 
%non-equivalent fusion categories of the same global dimension?
%\end{question}

\section{Preliminaries}
\subsection{Global dimension} \label{gldim}
For a fusion category $\C$ we will denote by $\O(\C)$ the set of isomorphism classes of simple objects
of $\C$. For an object $X\in \C$ let $\phantom{}^*X, X^*$ denote the right and left duals of $X$. Recall
that for any morphism $f: X\to X^{**}$ one defines its trace $\Tr(f)\in \BC$, see e.g. \cite[4.7]{EGNO}.

For any $X\in \O(\C)$ there exists a unique up to scaling isomorphism $a_X: X\to X^{**}$. A choice 
of $a_X$ gives an isomorphism $(a_X^*)^{-1}: X^*\to X^{***}=(X^*)^{**}$. It is clear that the following
scalar introduced by M.~M\"uger \cite{Mu} and called {\em squared norm} of $X$ in \cite{ENO}
is independent of the choice of $a_X$:
$$ |X|^2:=\Tr(a_X)\Tr((a_X^*)^{-1})\in \BC^\times .$$

Now the global dimension of $\C$ is defined as 
\begin{equation}\label{defdim}
\dim(\C)=\sum_{X\in \O(\C)}|X|^2.
\end{equation}

It was proved in \cite[Theorem 2.3 and Remark 2.5]{ENO} that $|X|^2$ is a totally positive algebraic
integer. It follows that
$\dim(\C)$ is a real algebraic integer $\ge 1$;
moreover $\dim(\C)=1$ if and only if $\C \simeq \Ve$. 

For any automorphism $\sigma$ of the field $\BC$ let $\C^\sigma$ be the twist of $\C$ by $\sigma$,
i.e. $\C^\sigma =\C$ as a category equipped with tensor product but the associativity constraint 
of $\C^\sigma$ is obtained from the one for $\C$ by applying $\sigma$. It is clear that
$\dim(\C^\sigma)=\sigma (\dim(\C))$. In particular the set $X$ in \eqref{fX} is Galois invariant.

\subsection{Spherical fusion categories} \label{pivot}
A {\em pivotal structure} on fusion category $\C$ is an
isomorphism of tensor functors $\delta: \id \to ()^{**}$, see e.g. \cite[4.7]{EGNO}, i.e.
functorial tensor isomorphism $\delta_X: X\to X^{**}$ for any $X\in \C$. In presence of such a structure
one defines {\em quantum dimensions} of objects $\dim_\delta (X)=\Tr(\delta_X)\in \BC$ which are
additive and multiplicative, see e.g. \cite[Proposition 4.7.12]{EGNO}.

A pivotal
structure $\delta$ is called {\em spherical} if $\dim_\delta (X)=\dim_\delta (X^*)$ for any $X\in \C$. 
A fusion category $\C$ is {\em sphericalizable} if it has at least one spherical structure.
All currently known fusion categories are sphericalizable.
A fusion category $\C$ equipped with a choice of spherical structure is called {\em spherical}. 
By abusing language we will consider these two notions as synonymous. Namely for
each sphericalizable fusion category we will choose and fix one spherical structure $\delta$;
we will often omit $\delta$ from notations, say by writing $\dim(X)$ instead of $\dim_\delta (X)$.

It is known that in a spherical fusion category $\dim(X)\in \BR$ for any $X\in \C$ and
$|X|^2=\dim(X)^2$ for any $X\in \O(\C)$, see \cite[Corollary 2.10]{ENO}. In particular
$$\dim(\C)=\sum_{X\in \O(\C)}\dim(X)^2.$$

\subsection{Formal codegrees} \label{fcgen}
Let $K(\C)$ be the Grothendieck ring of a fusion category $\C$ equipped with a basis $\{ b_i\}$ consisting of the classes of simple objects in $\C$. For any $b_i$ representing the isomorphism class of $L\in \C$ let $b_i'$ denote the class of the dual object $L^*$.

Let $E$ be an irreducible $\BC-$linear representation of the ring $K(\C)$. It was shown in 
\cite[Chapter 19]{Lu}
that the following element of $K(\C)\ot \BC$ is central:
$$\alpha_E=\sum_i\Tr(b_i,E)b_i'.$$
Moreover, the element $\alpha_E$ acts by some scalar $f_E$ on the irreducible representation $E$
and by zero on any irreducible representation not isomorphic to $E$. The scalars $f_E$ as $E$
runs over all irreducible representations of $K(\C)$ are called {\em formal codegrees} of $\C$,
see \cite{Ofc}. It is known that the formal codegrees are totally real algebraic integers which are
$\ge 1$, see \cite[Remark 2.12]{O3}. By \cite[Proposition 2.10]{O3} we have
\begin{equation}\label{sum1}
\sum_{E\in \Irr(K(\C))}\frac{\dim(E)}{f_E}=1.
\end{equation}

A one dimensional representation of $K(\C)$ is the same as ring homomorphism $\phi : K(\C)\to \BC$.
Such representation is automatically irreducible and the corresponding formal codegree $f_\phi$ is
given by
$$f_\phi =\sum_{X\in \O(\C)}\phi(X)\phi(X^*)=\sum_{X\in \O(\C)}|\phi(X)|^2$$
In particular, global dimension $\dim(\C)$ and its Galois conjugates are formal codegrees of $\C$.
Another example is given by the {\em Frobenius-Perron dimension} $\FPdim(\C)$ 
(see e.g. \cite[4.5]{EGNO}) and its Galois conjugates.

\section{Pseudo-unitary inequality without pseudo-unitarity}
\subsection{}
Let $f_1, f_2, \ldots ,f_r$ be the formal codegrees of $\C$. 
The following result was established in \cite[Theorem 2.21]{O3} for a pseudo-unitary $\C$ and was called
"pseudo-unitary inequality" there. This is unfortunate as we show now that pseudo-unitarity assumption is not needed.

\begin{theorem} For a spherical fusion category $\C$ the formal codegrees satisfy the following inequality:
\begin{equation}\label{pseudo}
\sum_{i=1}^r\frac1{f_i^2}\le \frac12\left(1+\frac1{\dim(\C)}\right).
\end{equation}
\end{theorem}

\begin{proof} It is known (see \cite[Theorem 2.13]{O3}) that the Drinfeld center $\Z(\C)$ contains simple
objects $Y_1, Y_2, \ldots , Y_r$ of dimensions $\dim(Y_i)=\frac{\dim(\C)}{f_i}$ (the objects 
$Y_i$ are precisely simple objects $Y$ of $\Z(\C)$ such that $F(Y)$ contains $\be$ with
nonzero multiplicity where $F: \Z(\C)\to \C$ is the forgetful functor). It is also known that
the twists of objects $Y_i$ equal 1, see \cite[Theorem 2.5]{O3}. 

Let $Z_1, Z_2, \ldots Z_s$ be the simple objects of $\Z(\C)$ distinct from $Y_i$; let
$\theta_1, \theta_2, \ldots \theta_s$ be their twists (recall that $\theta_i$ are roots of 1).

We have the following formulas for the global dimension and Gauss sum of $\Z(\C)$:
\begin{equation}\label{notgauss}
\dim(\C)^2=\dim(\Z(\C))=\sum_{i=1}^r\dim(Y_i)^2+\sum_{j=1}^s\dim(Z_j)^2,
\end{equation} 
\begin{equation}\label{gauss}
\dim(\C)=G(\Z(\C))=\sum_{i=1}^r\dim(Y_i)^2+\sum_{j=1}^s\theta_j\dim(Z_j)^2.
\end{equation} 
It follows from \eqref{gauss} that $\sum_{j=1}^s\theta_j\dim(Z_j)^2\in \BR$ and
\begin{equation}
\sum_{j=1}^s\theta_j\dim(Z_j)^2\ge -|\sum_{j=1}^s\theta_j\dim(Z_j)^2|\ge -\sum_{j=1}^s\dim(Z_j)^2.
\end{equation}
Thus adding \eqref{notgauss} and \eqref{gauss} we get
\begin{equation}
\dim(\C)^2+\dim(\C)=2\sum_{i=1}^r\dim(Y_i)^2+\sum_{j=1}^s\dim(Z_j)^2+\sum_{j=1}^s\theta_j\dim(Z_j)^2\ge 
\end{equation}
$$2\sum_{i=1}^r\dim(Y_i)^2=2\sum_{i=1}^r\frac{\dim(\C)^2}{f_i^2}$$

which is equivalent to desired inequality.
\end{proof}

\begin{remark} \label{Gaps}
By applying Galois automorphisms to the category $\C$ we see that we can replace
the right hand side of \eqref{pseudo} by $\frac12(1+\frac1f)$ where $f$ is arbitrary Galois conjugate
of $\dim(\C)$.
\end{remark}

It is natural to ask whether the inequality \eqref{pseudo} can be improved. The following example
shows that we cannot hope to improve constant $\frac12$ in \eqref{pseudo} by something better than 
$\frac14$ even if finitely many exceptions are allowed.

\begin{example} Let $G$ be a finite group and let $\C =\Rep(G)$ be the category of finite dimensional
$G-$modules. Then $K(\Rep(G))$ is commutative, so all representations of $K(\Rep(G))$ are one
dimensional. It is well known that all homomorphisms $K(\Rep(G))\to \BC$ are labelled by the conjugacy classes in $G$ and the orthogonality relations for the characters immediately imply that
the formal codegrees are $f_i=\frac{|G|}{|C_i|}$ where $C_i$ is a conjugacy class. Thus in this case\eqref{pseudo} is equivalent to
$$\sum_i|C_i|^2\le \frac12|G|(|G|+1).$$
Note that for the dihedral group $G=D_p$ with odd prime $p$ one has
$$\sum_i|C_i|^2=p^2+\frac12(p-1)\cdot 4+1=\frac14|G|^2+|G|-1.$$
Equivalently,
$$\sum_i\frac1{f_i^2}=\frac14+\frac1{\dim(\C)}-\frac1{\dim(\C)^2}.$$
\end{example}

\section{Bounds for formal codegrees}\label{bound}
\subsection{Global dimensions}
We start with a bound for the global dimension:

\begin{theorem} \label{gdim}
Let $\C$ be a spherical fusion category which is not equivalent to $\Ve$.

\emph{(i)} We have
\begin{equation} 
\dim(\C)>\frac43.
\end{equation} 

\emph{(ii)} Assume that $\dim(\C)$ has $k>1$
Galois conjugates. Then
\begin{equation}\label{kbound}
\dim(\C)\ge \sqrt{\frac{16k-16}{8k-7}}.
\end{equation} 
\end{theorem}

\begin{proof} (i) Recall that $\dim(\C)$ is an algebraic integer. Thus if $\dim(\C)\in \BQ$ then
$\dim(\C)\in \BZ$, so $\dim(\C)\ge 2$ and the bound holds.

Assume now that $\dim(\C)\not \in \BQ$. Thus there exists at least one Galois conjugate $f$ of $\dim(\C)$
distinct from $\dim(\C)$. Then by Remark \ref{Gaps} we have:
\begin{equation}
\frac1{\dim(\C)^2}+\frac1{f^2}\le \sum_{i=1}^r\frac1{f_i^2}\le \frac12(1+\frac1f),
\end{equation}
whence
\begin{equation}\label{916}
\frac1{\dim(\C)^2}\le \frac12+\frac12\cdot \frac1f-\frac1{f^2}=\frac9{16}-(\frac14-\frac1f)^2\le \frac9{16}.
\end{equation}
It follows that $\dim(\C)\ge \frac43$; however the equality is impossible since $\frac43$ is not an
algebraic integer.

(ii) Let $f$ be the largest of the conjugates of $\dim(\C)$. Then
\begin{equation}
\frac1{\dim(\C)^2}+(k-1)\frac1{f^2}\le \sum_{i=1}^r\frac1{f_i^2}\le \frac12(1+\frac1f),
\end{equation}
that is $\frac1{\dim(\C)^2}\le \frac12+\frac12\cdot \frac1f-(k-1)\frac1{f^2}$. The maximum value of
the right hand side is $\frac{8k-7}{16k-16}$ (achieved when $f=4(k-1)$), which implies the result.
\end{proof}

\begin{example}\label{rank2}
Let $K_n=\BZ[X]/\langle X^2=nX+1\rangle$ be based ring of rank 2. 
There are 2 homomorphisms $K_n\to \BC$ with formal codegrees
$\frac{n^2+4\pm n\sqrt{n^2+4}}2$ which are Galois conjugates if $n\ne 0$. 
One observes that $\frac{n^2+4-n\sqrt{n^2+4}}2<\frac43$ for $n>1$.
%It follows that any spherical categorification of $K_n$ is Galois conjugate of 
%pseudo-unitary category and Lemma \ref{psc} applies. However one shows
%that inequality \eqref{psc1} with $\gamma=\frac{n^2+4-n\sqrt{n^2+4}}2$ and
%$\FPdim(\C)=\frac{n^2+4+n\sqrt{n^2+4}}2$is violated for $n>1$. 
 Thus $K_n$ is not categorifiable by a spherical fusion category for $n>1$.
This falls short from the main result of \cite{O2} which states that $K_n$ is not categorifiable for $n>1$.
\end{example}

\begin{corollary}\label{kiz2} 
Let $d$ be a real number with $d<\sqrt{2}$. Then the set $X_s\cap [1,d]\subset \BR$ is finite.
\end{corollary}

\begin{proof} Assume $\dim(\C)\le d$. Then we get from inequality \eqref{kbound} 
$d\ge \sqrt{\frac{16k-16}{8k-7}}$ which gives an upper bound for $k$ as $\lim_{k\to \infty}\sqrt{\frac{16k-16}{8k-7}}=\sqrt{2}.$

Let $f$ be the largest of conjugates of $\dim(C)$. By Remark \ref{Gaps} we have 
$$\frac1{d^2}\le \frac1{\dim(\C)^2}\le \sum_{i=1}^r\frac1{f_i^2}\le \frac12(1+\frac1f),$$
that is $\frac1f\ge \frac2{d^2}-1$ which gives an upper bound for $f$.

Thus the minimal polynomial of algebraic integer $\dim(\C)$ has bounded degree and bounded coefficients (since the coefficients are elementary symmetric functions of the conjugates of $\dim(\C)$).
The result follows.
\end{proof}

%\begin{remark} One can improve Theorem \ref{gdim} as follows. Assume that $\dim(\C)$ has $k>1$
%Galois conjugates. Let $f$ be the largest of them. Then
%\begin{equation}
%\frac1{\dim(\C)^2}+(k-1)\frac1{f^2}\le \sum_{i=1}^r\frac1{f_i^2}\le \frac12(1+\frac1f),
%\end{equation}
%that is $\frac1{\dim(\C)^2}\le \frac12+\frac12\cdot \frac1f-(k-1)\frac1{f^2}$. By finding the maximum of
%the right hand side we get
%$$\dim(\C)\ge \sqrt{\frac{16k-16}{8k-7}}.$$
%This shows that for category $\C$ with $1<\dim(\C)<\dim(YL)$ we have $k\le 3$, cf. Question \ref{optim}.
%\end{remark}

\subsection{More bounds}
It is easy to generalize arguments in the proof of Theorem \ref{gdim} to get a bound for arbitrary
formal codegree of $\C$.

\begin{theorem} \label{codeg}
Let $\C$ be a spherical fusion category with $r$ simple objects. Then
any formal codegree $f_i$ of $\C$ satisfies
\begin{equation}\label{codef}
f_i\ge \sqrt{\frac{2r}{r+1}}.
\end{equation} 
In particular, if $\C \not \simeq \Ve$ then
\begin{equation} \label{coden}
f_i>\sqrt{\frac85}\approx 1.2649
\end{equation} 
\end{theorem}

\begin{proof}

We start proving Theorem \ref{codeg} with the following

\begin{lemma} \label{bigf}
Let $\C$ be a fusion category with $r$ simple objects. Then there exists
a Galois conjugate $f$ of $\dim(\C)$ such that $f\ge r$. 
\end{lemma}

\begin{proof} For an algebraic integer $a$ let $[a]\in \BQ$ be the average of the Galois conjugates of $a$.
It is easy to see that $a\mapsto [a]$ is a $\BQ-$linear functional on the $\BQ-$vector space of algebraic numbers. Also if $a$ is a totally positive algebraic integer then by the arithmetic and geometric mean
inequality one has $[a]\ge 1$. This applies to $a=|X|^2$ for $X\in \O(\C)$, see Section \ref{gldim}.
Let $f$ be the largest conjugate of $\dim(\C)$. Applying \eqref{defdim} we get
\begin{equation}\label{agm}
f\ge [\dim(\C)]=\sum_{X\in \O(\C)}[|X|^2]\ge \sum_{X\in \O(\C)}1=r.
\end{equation}
\end{proof}

\begin{remark}\label{bigfe}
 The equality $f=r$ in Lemma \ref{bigf} holds only in the case when $\C$ is a pointed
fusion category of dimension $r$. Namely in this case we see from \ref{agm} that $|X|^2=1$ 
for any $X\in \O(\C)$. This implies that the dimension of any simple object of the {\em pivotalization}
$\tilde \C$ of $\C$ (see \cite[Definition 7.21.9]{EGNO}) is $\pm 1$. This in turn implies that
the Frobenius-Perron dimensions of simple objects of $\tilde \C$ equal 1, see 
\cite[Exercise 9.6.2]{EGNO}. Thus $\tilde \C$, and hence $\C$, are pointed, see
\cite[Corollary 3.3.10]{EGNO}.
\end{remark}

Let us prove Theorem \ref{codeg}. Using Remark \ref{Gaps} and Lemma \ref{bigf} we have
\begin{equation}
\frac1{f_i^2}\le \sum_{i=1}^r\frac1{f_i^2}\le \frac12(1+\frac1f)\le \frac12(1+\frac1r)=\frac{r+1}{2r}.
\end{equation}
This implies \eqref{codef}. Thus \eqref{coden} holds for $r\ge 4$; in the remaining cases $r=2$
and $r=3$ one verifies \eqref{coden} case by case using classifications of fusion categories of rank $2$ (see \cite{O2} or Example \ref{rank2} below) and $3$ (see \cite{O3}).
\end{proof}

\begin{corollary}\label{cFP}
 Any Galois conjugate of the Frobenius-Perron dimension $\FPdim(\C)$
of a nontrivial spherical fusion category $\C$ is greater than $\sqrt{\frac85}$.
\end{corollary}

We note that Theorem \ref{codeg} and Corollary \ref{cFP} are of interest even for {\em pseudo-unitary} (see e.g. \cite[9.4]{EGNO}) fusion categories while Theorem \ref{main} is trivial for such categories.
We have a stronger version of Corollary \ref{cFP} for such categories:

\begin{lemma}\label{psc}
 Let $\C$ be a Galois conjugate of pseudo-unitary fusion category. Then any conjugate
$\gamma$ of the Frobenius-Perron dimension $\FPdim(\C)$ satisfies
\begin{equation}\label{psc1}
\gamma \ge \sqrt{\frac{2\FPdim(\C)}{\FPdim(\C)+1}}.
\end{equation}
\end{lemma}

\begin{proof} We can assume that $\C$ is pseudo-unitary. Recall that $\dim(\C)=\FPdim(\C)$ (see
\cite[Definition 9.4.4]{EGNO}) and $\C$ is automatically spherical (see \cite[9.5]{EGNO}). Thus
\eqref{pseudo} gives
$$\frac1{\gamma^2}\le \sum_{i=1}^r\frac1{f_i^2}\le \frac12(1+\frac1{\dim(\C)})=\frac12(1+\frac1{\FPdim(\C)})=\frac{\FPdim(\C)+1}{2\FPdim(\C)}.$$
The result follows.
\end{proof}

It would be interesting to improve Lemma \ref{psc}. For instance one can ask whether there exists 
a bound for $\gamma$ which is unbounded function of $\FPdim(\C)$.

\section{Fusion categories of the same global dimension}
\subsection{} Here is the main result of this Section:
\begin{theorem}\label{ffib}
 There exists no infinite family of pairwise non-equivalent fusion categories
of the same global dimension.
\end{theorem}

\begin{proof} Let $d$ be a totally positive algebraic integer and let $\C$ be a fusion category 
with $\dim(\C)=d$. Let $f$ be the largest conjugate of $f$. Then by Lemma \ref{bigf}
we have that $r=|\O(\C)|\le \lfloor f\rfloor$. The Grothendieck ring $K(\C)$ has rank $r$, so it has at most
$r$ homomorphisms to $\BC$. Hence $\FPdim(\C)$ has at most $r$ conjugates. 

It is known that for any conjugate $\gamma$ of $\FPdim(\C)$ the ratio $\frac{d}{\gamma}$ is an
algebraic integer, see \cite[Corollary 2.14]{O3}. 
%(this can proved in the same way as \cite[Proposition 9.4.2]{EGNO}; alternatively
%combine {\em loc. cit.} and \cite[]{Ofc}). 
Thus the norm $N$ (i.e. product of all conjugates) of $\FPdim(\C)$ divides $d^k$ where $k$ is the number of the conjugates.  Since $k\le r\le \lfloor f\rfloor$ we see that 
$N$ divides $d^{\lfloor f\rfloor}$. Thus $N\le M$ where $M$ is the largest integer dividing 
$d^{\lfloor f\rfloor}$. Finally $\FPdim(\C)\le N$ since all conjugates of $\FPdim(\C)$ are $\ge 1$,
see Section \ref{fcgen}. We proved that
$$\FPdim(\C)\le M=\mbox{largest integer dividing}\; \; d^{\lfloor f\rfloor}.$$
The theorem is proved as there are just finitely many fusion categories of bounded Frobenius-Perron 
dimension (namely it is easy to see that there are just finitely many based rings of bounded
Frobenius-Perron dimension and for each such ring there are just finitely many categorifications
by \cite[Theorem 9.1.4]{EGNO}).
\end{proof}

One can try to classify all fusion categories with $\dim(\C)=d$ where $d$ is small number. Here are
some examples.

\begin{example}\label{smalld} 
(i) Let $d=2$. Then by Lemma \ref{bigf} we have $r\le 2$; clearly $r=1$ is impossible
and for $r=2$ we get only pointed categories associated with the group $\BZ/2\BZ$ by Remark
\ref{bigfe}.

(ii) Let $d=3$. In this case $r\le 3$; $r=1$ is impossible and $r=3$ gives pointed categories. In the case
$r=2$ by the arguments in the proof of Theorem \ref{ffib} we see that the norm of the Frobenius-Perron
dimension divides $d^r=9$. However one computes easily that the norm of the Frobenius-Perron dimension of the ring $K_n$ (see Example \ref{rank2}) is $n^2+4$ which never divides 9. Thus this
case is impossible.

(iii) Let $d=4$ and assume that category $\C$ is spherical and not pointed. Thus $r\le 3$ and we 
have at most 3 distinct homomorphisms $K(\C)\to \BC$. Consider the action of the Galois group
on such homomorphisms. If the dimension homomorphism and the Frobenius-Perron homomorphism
are in the same orbit then the category is weakly integral (see \cite[9.6]{EGNO}), hence {\em Ising category} (see e.g. \cite[Appendix B]{DGNO}). Otherwise either orbit of the dimension or of the 
Frobenius-Perron dimension has exactly one element; this again implies that $\C$ is weakly integral,
see \cite[Exercise 9.6.2]{EGNO}.

(iv) Let $d=\frac{5\pm \sqrt{5}}2$ and assume that category $\C$ is spherical.
Then the Galois
conjugate $f=\frac{5\mp \sqrt{5}}2$ is also one of formal codegrees of $\C$. Since
$$\frac1{\dim(\C)}+\frac1f=1,$$
equation \eqref{sum1} implies that $\C$ has no other formal codegrees. Thus $\C$ is of rank 2, and
using Example \ref{rank2} we see that $K(\C)=K_1$. It follows that $\C \simeq YL$ or
$\C \simeq \overline{YL}$ where $YL$ is the Yang-Lee
category and $\overline{YL}$ is its Galois conjugate depending on the value of $\dim(\C)$, see e.g. 
\cite[2.5]{O2} and \cite[Exercise 8.18.7]{EGNO}.

(v) Let $d=5$ and assume that category $\C$ is spherical. Hence $r\le 5$. 
In the case $r\le 3$ the arguments as in (iii) above show that $\C$ is weakly integral, hence pointed,
see \cite[Corollary 8.30]{ENO}. The same conclusion holds if $r=5$ by Remark \ref{bigfe}.

We assume now that $r=4$ and $\C$ is not weakly integral. Then there are 4 homomorphisms
$K(\C)\to \BC$ which split into two orbits of size 2 under the action of the Galois group: orbit of the
dimension homomorphism and orbit of the Frobenius-Perron homomorphism. Both formal codegrees
for the first orbit equal $\dim(\C)=5$; let $f_1$ and $f_2$ be the formal codegrees for the second orbit.
Then $f_1f_2$ ($=$ norm of the Frobenius-Perron dimension) is an integer $>1$ and dividing $d^2=25$.
Also $\frac1{f_1}+\frac1{f_2}=1-\frac25=\frac35$ by \eqref{sum1}. In the case $f_1f_2=5$
we get $f_1+f_2=3$ whence $f_1$ and $f_2$ are not real. Thus $f_1f_2=25$ and $f_1+f_2=15$
whence $\FPdim(\C)=\frac{15+5\sqrt{5}}2=\left(\frac{5+\sqrt{5}}2\right)^2$. We claim now that
$\C$ must be equivalent to $YL\boxtimes \overline{YL}$, cf. \cite[Exercise 9.4.6]{EGNO}.
Indeed, there is a unique expansion of $\FPdim(\C)-1=\frac{13+5\sqrt{5}}2$ into a sum of
three totally positive algebraic integers from the field $\BQ(\sqrt{5})$ and maximal among their
conjugates:
$$\frac{13+5\sqrt{5}}2=\frac{3+\sqrt{5}}2+\frac{3+\sqrt{5}}2+\frac{7+3\sqrt{5}}2=\left( \frac{1+\sqrt{5}}2\right)^2+\left( \frac{1+\sqrt{5}}2\right)^2+\left( \frac{3+\sqrt{5}}2\right)^2.$$
It follows that the simple objects of $\C$ have the Frobenius-Perron dimensions $1, \frac{1+\sqrt{5}}2,$ 
$\frac{1+\sqrt{5}}2, \frac{3+\sqrt{5}}2$. Let $X$ and $Y$ be the simple objects of $\C$ with 
$\FPdim(X)=\FPdim(Y)=\frac{1+\sqrt{5}}2$. Since $X\otimes X^*$ is a direct sum of $\be$ and an
object of dimension $\frac{1+\sqrt{5}}2$, we see that at least one of $X$ and $Y$ is self dual.
Hence both of them are self dual and $X\otimes Y\simeq Y\otimes X$ is simple. Thus 
$\Hom(X\otimes X, Y)=\Hom(X, Y\otimes X)=0$, whence $X\otimes X=\be \oplus X$ and, similarly,
$Y\otimes Y\simeq \be \oplus Y$. In other words we have an isomorphism of based rings $K(\C)=K(YL)\otimes K(YL)=K(YL \bt \overline{YL})$.

Let us construct a tensor equivalence $\C \simeq YL \bt \overline{YL}$. We can assume that
$\dim(X)=\frac{1-\sqrt{5}}2$ and $\dim(Y)=\frac{1+\sqrt{5}}2$. Let $YL\subset \C$ be the fusion subcategory consisting of direct sums of $\be$ and $X$ and let $\overline{YL}\subset \C$ be the
fusion subcategory consisting of direct sums of $\be$ and $Y$. Consider the Drinfeld center
$\Z(\C)$ of $\C$, see e.g. \cite[7.13]{EGNO}.
%Let $\Z (\C)$ be the Drinfeld
%center of $\C$. Note that $\Z(\C)$ contains no nontrivial symmetric subcategories since it has
%no nontrivial simple objects of integral Frobenius-Perron dimension, see \cite[9.9]{EGNO}.
Using \cite[Proposition 9.2.2]{EGNO} and \cite[Theorem 2.13]{O3} we deduce
that $\Z(\C)$ contains simple objects which map to $\be \oplus X$ and $\be \oplus Y$ under the
forgetful functor $\Z(\C)\to \C$; moreover these objects have twists 1. 
Thus the subcategories $\A \subset \Z(\C)$ and $\B \subset \Z(\C)$
consisting of objects sent to $YL$ and $\overline{YL}$ under the forgetful functor are non-trivial.
We have obvious braided tensor functors $\A \to \Z(YL)$ and $\B \to \Z(\overline{YL})$. Recall
that $YL$ has a structure of the modular category (see \cite[8.18]{EGNO}) whence
$\Z(YL)\simeq YL\bt YL$ as fusion categories. It is easy to see that no proper subcategory
of $\Z(YL)$ contains an object with twist 1 and which maps to $\be \oplus X$ under the forgetful
functor. We conclude that the functor $\A \to \Z(YL)$ is surjective and, similarly, the functor
$\B\to \Z(\overline{YL})$ is surjective. Comparing the Frobenius-Perron dimensions (see e.g.
\cite[Lemma 3.38]{DGNO}) we conclude
that both functors are equivalences and $\Z(\C)\simeq \A \bt \B \simeq YL \bt YL \bt \overline{YL}
\bt \overline{YL}$. Hence $\Z(\C)$ contains 4 fusion subcategories equivalent to $YL\bt \overline{YL}$
and such that the restriction of the forgetful functor yields a desired equivalence 
$\C \simeq YL\bt \overline{YL}$.
\end{example}

\appendix

\section{}
%\centerline{\bf Appendix} 

%\vskip .05in

\centerline{\bf By Pavel Etingof and Victor Ostrik}

\vskip .05in

\subsection{}
The goal of this appendix is 
%to use the results by V. Ostrik in the body of the paper 
to prove (using a computer) the following proposition. 

\begin{proposition}\label{mainpro} The only number $1<d<4\sqrt{3}/5=1.38564...$ which can serve as the 
global dimension of a spherical fusion category over complex numbers is $d=\frac{5-\sqrt{5}}{2}=1.381966..$. 
\end{proposition} 

\begin{proof} Let $1<d<\sqrt{2}$ be the global dimension of a spherical fusion category (a totally positive algebraic integer). Let $k$ be the number of algebraic conjugates of $d$. By Theorem \ref{gdim} (i) we know that $d>\frac43$ and Theorem \ref{gdim} (ii) implies that the Proposition holds for $k\ge 4$, so
it suffices to prove it for $k=2,3$. 

%In the body of the paper, it is shown that $d>4/3$ and that the proposition holds for $k\ge 4$, so it suffices to prove it for $k=2,3$. 

Let $d_*$ be the biggest conjugate of $d$. By equation \eqref{916} we have
%The "pseudounitary inequality" in the body of the paper 
%implies that 
$$
\frac{1}{d^2}\le \frac{9}{16}-\left(\frac{1}{4}-\frac{1}{d^*}\right)^2,
$$
which yields
\begin{equation}\label{mainineq}
d^*\le K,\text{ where }
K=K(d):=\left(\frac{1}{4}-\sqrt{\frac{9}{16}-\frac{1}{d^2}}\right)^{-1}. 
\end{equation}

Let $k=2$. Then we have 
$$
(3d-4)(3d_*-4)\ge 1, 
$$
so by \eqref{mainineq},
$$
(3d-4)(3K(d)-4)\ge 1,
$$
which yields
$$
d\ge \frac{1}{4}(\sqrt{41}-1)=1.3507...
$$

Now we can use a simple computer program to show that there are no possible values of $d$ in the range 
$$
\frac{1}{4}(\sqrt{41}-1)=d_-\le d<d_+=\frac{4}{5}\sqrt{3}.
$$ 
Namely, let 
$$
P(x)=x^2-ax+b
$$
be the minimal polynomial of $d$ over $\Bbb Z$. Then we have $P(d_-)\ge 0$ and $P(d_+)<0$, 
so 
$$
ad_--d_-^2\le b<ad_+-d_+^2.
$$
Also by \cite[Theorem 1.2]{Ofc}, $d$ has to be a $d$-number, i.e., must divide all its conjugates, 
which results in the condition that $a^2$ is divisible by $b$. 
Finally, we have 
$$
a=d+d^*\le \frac{4}{5}\sqrt{3}+K\left(\frac{4}{5}\sqrt{3}\right),
$$
which yields $a\le 23$. Going through all possibilities, we find that there are no solutions except 
$a=b=5$, which gives $d=(5-\sqrt{5})/2$. (In fact, there are so few cases that this can be easily checked by hand). 

Now suppose that $k=3$, and $d<d_+$. By Theorem \ref{gdim} (ii) we have $d\ge \tilde d_-=
\sqrt{\frac{32}{17}}\approx 1.37199$. Let
$$
P(x)=x^3-ax^2+bx-c
$$
be the minimal polynomial of $d$. Without loss of generality we may assume that $d$ is the smallest root of $P$. Let $d_1=d<d_2<d_3=d_*$ be the roots of $P$ (they all have to be real and positive). Let us list some more properties of $(a,b,c)$: 

(i) $P(\tilde d_-)\le 0$ and $P(d_+)>0$ ($P(x)$ cannot have 2 roots $<2$ since it would contradict
to \eqref{sum1}), so 
$$
b\tilde d_--a\tilde d_-^2+\tilde d_-^3\le c<bd_+-ad_+^2+d_+^3;
$$

(ii) $d_3=d_*\le K(\frac{4}{5}\sqrt{3})\le 22$, thus $a=d_1+d_2+d_3\le 45$. 

(iii) $b\le a^2/3$ (by the arithmetic and geometric mean inequality, as all the roots are real). 

(iv) $d$ has to be a $d$-number, hence $c$ divides $a^3$ and $c^2$ divides $b^3$. 

(v) Inequality \eqref{mainineq} is satisfied. 

Again, going through all possibilities, we find that there are no solutions; the program (in MAGMA) 
does this for less than a second. 

This completes the proof of Proposition \ref{mainpro}. 
\end{proof} 

\begin{remark} Recall that by Example \ref{smalld} (iv) the only spherical fusion category with 
$\dim(\C)=\frac{5-\sqrt{5}}2$ is the Yang-Lee category $YL$.
\end{remark}

\bibliographystyle{ams-alpha}

\end{document}